\documentclass{amsart}
\usepackage{amsfonts}
\usepackage{amsmath}
\usepackage{amssymb}
\usepackage{amsthm}
\usepackage{multirow}
\usepackage{graphics}
\usepackage{graphicx}
\usepackage{float}
\usepackage{adjustbox}
\usepackage{enumerate}
\usepackage{MnSymbol}
\usepackage[pagewise,mathlines]{lineno}

\numberwithin{equation}{section}

\newcommand{\Z}{\mathbb{Z}}
\newcommand{\N}{\mathbb{N}}
\newcommand{\Fp}{\mathbb{F}_{p}}
\newcommand{\Zp}{\mathbb{Z}_{p}}

\newcommand{\om}{\omega_{o}}

\newtheorem{theorem}{Theorem}[section]

\newtheorem{lemma}[theorem]{Lemma}
\newtheorem{corollary}[theorem]{Corollary}

\newtheorem{example}[theorem]{Example}

\theoremstyle{definition}
\newtheorem{definition}[theorem]{Definition}

\begin{document}


\title[An Analogue of the Davenport Constant]{Quadratic Symmetric Polynomials and an analogue of the Davenport Constant}

\author{H. Godinho, A. Lemos$^{\ast}$, V.G.L. Neumann$^{\ast}$ \and F.A.A. Oliveira}

\thanks{$\ast$ The second author was partially supported by FAPEMIG grant APQ-02546-21 and FAPEMIG grant RED-00133-21 and third author was partially supported by FAPEMIG grant APQ-03518-18 and FAPEMIG grant RED-00133-21.}

\address{Departamento de Matem\'{a}tica, Universidade de Bras\'ilia, Bras\'ilia-DF, Brazil}

\email{hemar@mat.unb.br}

\address{Departamento de Matem\'{a}tica, Universidade Federal de Vi\c cosa, Vi\c cosa-MG, Brazil}

\email{abiliolemos@ufv.com.br}

\address{Departamento de Matem\'{a}tica, Universidade Federal de Urbelândia, Uberlândia-MG, Brazil}

\email{victor.neumann@ufu.br}

\address{Instituto Federal de Educa\c c\~ao, Ci\^encia e Tecnologia do Rio Grande do Sul-Campus Farroupilha, Farroupilha-RS, Brazil}

\email{filipe.oliveira@farroupilha.ifrs.edu.br}

\subjclass[2010]{11P70, 11B30}

\keywords{Davenport constant, sequences, symmetric polynomials}
\begin{abstract}

In this paper, we define the constant $D(\varphi, p)$, an analogue for the Davenport constant,  for sequences on the finite field $\mathbb{F}_p$,  defined via quadratic symmetric polynomials. Next, we state a series of results presenting either the exact value of $D(\varphi, p)$, or lower and upper bounds for this constant.
\end{abstract}

\maketitle

\section{Introduction}

Let $p$ be a prime, and denote by $\Fp$ the field  of $p$ elements, and $\Fp^{*} = \Fp\backslash \{0\}$. A sequence $S=a_{1}a_{2}\cdots a_{m}$ of length $m$ in $\Fp$ will also be written as
\begin{equation}
S=[u_1]^{n_1}[u_2]^{n_2}\cdots[u_t]^{n_t} := \underbrace{u_{1}u_{1} \cdots u_{1}}_{n_1\mbox{ {\scriptsize times}}}\underbrace{u_{2}u_{2}\cdots u_{2}}_{n_2\mbox{ {\scriptsize times}}}\cdots\underbrace{u_{t}u_{t}\cdots u_{t}}_{n_t\mbox{ {\scriptsize times}}},
\end{equation}
where  $u_{1},\ldots,u_{t}$ are distinct elements of $\Fp$  and  $m= n_{1}+n_{2}+\cdots + n_{t}$.  A subsequence $T$ of $S$ is a sequence of the form
\[
T=[u_1]^{s_1}[u_2]^{s_2}\cdots[u_t]^{s_t}, \;\;\mbox{with} \;\;\; 0 \leq s_{j}\leq n_{j}, \;\;\; j=1,\ldots, t,
\]
where $s_{j}=0$ denotes that the element $u_{j}$ does not appear in $T$.

Let  $s_{n, k}$ be the symmetric polynomial
\begin{equation}
s_{n,k} := s_{n,k}(x_1, x_2, \ldots , x_n) = x_1^k + x_2^k + \cdots + x_n^k,
\end{equation}
in $n$ variables and degree $k$.  For any polynomial $F(y_{1},\ldots, y_{r}) \in\Fp[y_{1},\ldots, y_{r}]$ define the symmetric polynomial in $n$ variables
\[
\varphi(F,n)=F(s_{n,1},s_{n,2},\ldots,s_{n,r}) \in \Fp[x_1, x_2, \ldots, x_n] .
\]
\begin{definition}
Let $F$ be a polynomial in $\Fp[y_{1}, \ldots, y_{r}]$ and $S$ be a sequence of length $n$  in $\Fp$ denoted by $S = a_1 a_2 \cdots a_{n}$. We  define 
\[
\varphi(S):=F(s_{n,1},s_{n,2},\ldots,s_{n,r})(a_1,a_2,\ldots, a_n).
\]
Observe that  the order of appearance of an element in the sequence is not relevant for the evaluation of  $\varphi(S)$.
\end{definition}
\begin{example}
Let $F(y_{1},y_{2}) = ay_{1}^{2} + by_{1}y_{2}+ cy_{2}^{4}$,  and $S=a_{1}a_{2}a_{3}a_{4}$. Then
\[ \varphi(F,4)\; = a(s_{4,1})^{2} + b(s_{4,1})\cdot (s_{4,2}) + c(s_{4,2})^{4},\]  \[\mbox{and}\] 
\[
\varphi(S) = a(a_{1}+a_{2}+a_{3}+a_{4})^{2} + b(a_{1}+a_{2}+a_{3}+a_{4})(a_{1}^{2}+a_{2}^{2}+a_{3}^{2}+a_{4}^{2}) + c(a_{1}^{2}+a_{2}^{2}+a_{3}^{2}+a_{4}^{2})^{4}.
\]
\end{example}
\begin{definition}
Let $F$ be a polynomial in $\Fp[y_{1}, \ldots, y_{r}]$ and $S$ be a sequence of length $m$ in $\Fp$ denoted by $S = a_1 a_2 \cdots a_{m}$. 
\begin{enumerate}
\item[(i)] The sequence $S$  will be called a \textit{$\varphi$-zero sequence} if $\varphi(S) =0$.

\item[(ii)] The sequence $S$ will be called  a \textit{$\varphi$-zero free sequence} if for any subsequence $T= b_{1}\cdots b_{r}$ of $S$ we have $\varphi(T) \neq 0.$

\end{enumerate}
\end{definition}

\begin{definition} The Davenport  $\varphi$-Constant $D(\varphi, p)$ is defined as the smallest value $\ell$ such that any sequence in $\Fp$ of length at least $\ell$ has a $\varphi$-zero subsequence.
\end{definition}

\begin{definition}
Let $F$ be a polynomial in $\Fp[y_{1}, \ldots, y_{r}]$ and $S$ be a sequence of length $m$ in $\Fp$ denoted by $S = a_1 a_2 \cdots a_{m}$. The sequence $S$ will be called  an \textit{extremal $\varphi$-zero free sequence} if for any nonempty subsequence $T= b_{1}\cdots b_{r}$ of $S$ we have $\varphi(T) \neq 0$ and $m=D(\varphi, p)-1.$
\end{definition}

 The study of zero-sum sequences over abelian groups  is a very active and beautiful area of research in Additive Number Theory,
and the concept of \textit{$\varphi$-zero sequence} is an extension of these classical ideas, just by taking into the definitions above  the polynomial $F(y_{1})=y_{1}$, for in this case  $\varphi(F,n)=s_{n,1}= x_{1}+\cdots +x_{n}$ and $D(\varphi,\Fp) = D(\Zp)$, the classical Davenport constant over $\Zp$,  the additive subgroup of $\Fp$.  In this context, a $\varphi$-zero sequence is called a {\em zero-sum sequence}, a $\varphi$-zero free sequence is called a {\em zero-sum free sequence} and an extremal $\varphi$-zero free sequence is called an {\em extremal zero-sum free sequence}. For our purposes here in this paper, we need the following  consequence of the results given in  Olson\cite{OL1}.
\begin{theorem}\label{olson} Let $p$ be a prime. Then
\begin{enumerate}
\item $D(\mathbb{Z}_p)=p$ and any extremal zero-sum free sequence in $\mathbb{Z}_p$  is of the form $[u]^{p-1}$, for $u\in\mathbb{Z}_p\backslash\{0\}$. \\
\item $D(\mathbb{Z}_p \oplus \mathbb{Z}_p) = 2p-1$.
\end{enumerate}
\end{theorem}

We refer the interested reader to \cite{GAO} and \cite{GAH} for more information on  zero-sum sequences over abelian groups. 

 Our work was very much inspired by  Bialostocki and Luong\cite{BIA}, who presented  a generalization of the
 Erd\H os-Ginzburg-Ziv Constant (EGZ-Constant), via quadratic symmetric polynomials. Our goal is to present three theorems describing  bounds for $D(\varphi,\Z_{p})$ for the  specific classes of quadratic polynomials, and also make whenever possible,  a complete list of all extremal $\varphi$-zero free sequences. For this purpose we define $M(\varphi, p)$ as the set of all extremal $\varphi$-zero free sequences in $\Fp.$ In almost all cases covered by Theorem \ref{Principal}, we were able to give a complete answer for the questions presented. 
 The results given in Theorems \ref{Theo3.1}, \ref{lb} and \ref{cota},  although simpler to state, were more challenging and reveled an interesting relation between $D(\varphi,p)$ and  the distribution of quadratic residues modulo $p$. This research is in its initial stage, but already reveals some interesting questions  waiting to be answered. We close this paper with some comments and  questions, indicating possible directions for the development of the theory. 

\section{Quadratic Symmetric Polynomials over $\Fp$}

Let $F(y_{1},y_{2}) = ay_{1}^{2} + by_{2} + cy_{1}$, in this case
\begin{equation}\label{quad}
\varphi(F,m) = F(s_{m,1}, s_{m,2}) = a s_{m,1}^2 + b s_{m,2} + cs_{m,1}.
\end{equation}

We are interested in symmetric quadratic polynomials, so we will always assume that either $a$ or $b$ are not zero in \eqref{quad}.
Next we present our first result.
\begin{theorem} \label{Principal} Let $p$ be an odd prime. Then the following statements are true:
\begin{enumerate}
\item[(i)] If $a = 0$ and $b \neq 0$, then $D(\varphi ,p) = p$ and 
\[
M(\varphi, p) = \{[u]^{\alpha}[-u-cb^{-1}]^{p-1-\alpha}\; / \; u\in \Fp\backslash\{0, -cb^{-1}\} \mbox{ and } 0 \le \alpha \le p-1\}.\]
\item[(ii)] If $a \neq 0$ and $b = c = 0$, then $D(\varphi , p) = p$ and
\[
M(\varphi, p) = \{[u]^{p-1}\; / \; u\in \Fp^{*}\}.
\]
\item[(iii)]  If $a \neq 0$, $b = 0$ and $c \neq 0$, then 
\[
D(\varphi , p) = p-1 \;\;\; \mbox{and} \;\;\;M(\varphi,p) = \{[ca^{-1}]^{p-2}\}.
\]
\item[(iv)] If $ab\neq 0$ then $D(\varphi , p ) \leq 2p-1$.
\end{enumerate}
\end{theorem}

The proof of this Theorem will be presented in the form of a series of three lemmas.
\begin{lemma}
 Let $p$ be an odd prime. Then the following statements are true:
\begin{enumerate}
\item[(i)] If $a = 0$ and $b \neq 0$, then $D(\varphi ,p) = p$ and 
\[
M(\varphi, p) = \{[u]^{\alpha}[-u-cb^{-1}]^{p-1-\alpha}\; / \; u\in \Fp\backslash\{0, -cb^{-1}\} \mbox{ and } 0 \le \alpha \le p-1\}.\]
\item[(ii)] If $a \neq 0$ and $b = c = 0$, then $D(\varphi , p) = p$ and
\[
M(\varphi, p) = \{[u]^{p-1}\; / \; u\in \Fp^{*}\}.
\]
\end{enumerate}
\end{lemma}
\begin{proof} Initially suppose $a = 0$ and $b \neq 0$, then it follows from \eqref{quad} that for all $k \in \N$, 
\begin{equation}\label{fs}
\varphi(F,k)= \sum_{i=1}^{k} (bx_i^2 + cx_i) = \sum^{r}_{i=1} f(x_{i}), \;\;\; \mbox{with} \;\;\; f(u) = bu^{2} + cu.
\end{equation}
Given a sequence $T = u_1u_2 \cdots u_k$ in $\Fp$, consider the sequence $S = f(u_1) f(u_2) \cdots f(u_k)$  in $\Zp$. From \eqref{fs} and Theorem \ref{olson} it follows that $D(\varphi,p) \leq D(\Zp) = p$, and if $S$ is an extremal zero-sum free sequence in $\Zp$ then $S=[f(u)]^{p-1}$. Since 
\[f(u_{i}) = f(u_{j}) \quad \Leftrightarrow \quad u_{i} = u_{j} \, \mbox{ or } \, u_{j}= -u_{i} -cb^{-1},\]
we have that  $M(\varphi,p) = \{ [u]^{\alpha}[-u-cb^{-1}]^{p-1-\alpha}, \mbox{ where } u\in \Fp\setminus\{0, -cb^{-1}\}$.

If $a \neq 0$ and $b = c = 0$, then it follows that \eqref{quad} has the form 
\[
\varphi(F,k) = a s_{k,1}^2.
\]

 From Theorem \ref{olson}  it is easy to conclude that $D(\varphi , p ) = D(\Zp)= p$ and 
\[
M(\varphi,p) = \{[u]^{p-1} \, \, ; \, \, u\in \Fp\setminus\{0\}\}.
\]
\end{proof}
\begin{lemma}Let $p$ be an odd prime. If $a \neq 0$, $b = 0$ and $c \neq 0$, then 
\[
D(\varphi , p) = p-1 \;\;\; \mbox{and} \;\;\;M(\varphi,p) = \{[ca^{-1}]^{p-2}\}.
\]
\end{lemma}
\begin{proof}
Now let us assume  $ac \neq 0$ and $b = 0$, and we will have
\begin{equation}\label{C3}
\varphi(F,k) = a s_{k,1}^2 + cs_{k,1} =a(x_{1} + \cdots + x_{k})^{2} + c(x_{1} + \cdots + x_{k}).
\end{equation}

Hence  for any  zero-sum sequence $S$ in $\Zp$, we have $\varphi(S)=0$, therefore  it follows from Theorem \ref{olson} that  $D(\varphi,p) \leq D(\Zp)= p$. Let us suppose that we can find a  $\varphi$-zero free sequence $S_{0}=u_{1}u_{2}\cdots u_{p-1}$ of length $p-1$. In particular, $S_{0}$ must also be   a zero-sum free sequence in $\Z_{p}$ (see  \eqref{C3}), then it follows from Theorem \ref{olson} that $S_{0} = [u]^{p-1}$, for some $u\in\Zp\backslash\{0\}$.
Let $T_{j} =  [u]^{j}$ be a subsequence of $S_{0}$. Considering $T_{j}$ also a sequence in $\Fp$ we would have
\[
\varphi(T_{j}) = a(ju)^{2} + c(ju) = ju(aju +c), 
\]
and it is always possible to find a $j$ such that $\varphi(T_{j}) =0$, contradicting the hypothesis that $S_{0}$ is a $\varphi$-zero free sequence. Therefore $D(\varphi,p) \leq p-1$.

Now consider the sequence $S=[ca^{-1}]^{p-2}$ and observe that, for any $1\leq j  \leq p-2$ we have
\[
\varphi([ca^{-1}]^{j}) = jc^{2}a^{-1}(j+1) \neq 0,
\]
thus $S$ is a $\varphi$-zero free sequence, proving that  $D(\varphi,p)= p-1$.

Suppose $S=w_{1}w_{2}\cdots w_{p-2}$ is also a  $\varphi$-zero free sequence in $\Fp$. Then for any subsequence
 $T=v_{1}\cdots v_{t}$ of $S$ we have (see \eqref{C3})
\[
v_{1}+\cdots+ v_{t}\not\equiv 0 \pmod{p}\;\;\;\;\mbox{and}\;\; \;\; v_{1}+\cdots+ v_{t}\not\equiv -c\cdot a^{-1} \pmod{p}.
\]
In particular the sequence, with $v_{0}=\,ca^{-1}$, 
\[
w_{1}w_{2}\cdots w_{p-2}v_{0}
\]
is a zero-sum free sequence in $\Zp$. Again it follows from Theorem \ref{olson} that
\[
w_{1}=w_{2}=\cdots = w_{p-2}=\,ca^{-1},
\]
completing the proof.
\end{proof}
\begin{lemma}\label{dif} 
 If   $ab\neq 0$ then $D(\varphi ,p ) \leq 2p-1$.
\end{lemma}
\begin{proof} Let $S=u_{1}u_{2}\cdots u_{m}$ be any sequence in $\Fp$ of length $m\geq 2p-1$, and consider the sequence 
\[U = [u_1,u_1^2] [u_2,u_2^2] \cdots  [u_{m},u_{m}^2] \, \, \mbox{ over } \, \, \Zp\oplus \Zp.\]
According to Theorem \ref{olson}, since  $m\geq 2p-1$,  we can find a  subsequence $V$ of $U$ of length $k \leq 2p-1$, say
\[V=[v_{1},v_{1}^2][v_{2},v_{2}^2]\cdots [v_{k},v_{k}^2],\]
such that, modulo $p$ we have,
\[ 
v_{1}+ v_{2}+ \cdots + v_{k}\equiv 0 \;\;\;\mbox{and} \;\;\; v_{1}^{2}+ v_{2}^{2}+ \cdots + v_{k}^{2}\equiv 0
\]
hence 
\[\varphi(V)=a(v_{1}+ v_{2}+ \cdots + v_{k})^{2} + b( v_{1}^{2}+ v_{2}^{2}+ \cdots + v_{k}^{2}) + c(v_{1}+ v_{2}+ \cdots + v_{k}) = 0
\]
 in $\Fp$, and this proves that  $D(\varphi ,p ) \leq 2p-1$. 
\end{proof}

\section{ The Polynomial  $\varphi(F,m) = as_{m,1}^{2} +bs_{m,2} $, with $ab\neq 0$}

In this section we are interested in the symmetric polynomial  $\varphi(F,m) = as_{m,1}^{2} +bs_{m,2} $, with $ab\neq 0$. With no loss in generality, we start by  rewriting the polynomial $\varphi(F,p)$ as 
\begin{equation}\label{mu=0}
\varphi(F_{\lambda},k) =  s_{k,1}^2 + \lambda s_{k,2}\, , \;\;\;\mbox{with}\;\;\; \lambda\neq 0. 
\end{equation}
 This is  the main result of this section.
\begin{theorem}\label{Theo3.1} Let $p$ be an odd prime, then 
\[
  D(\varphi , p)=1, \;\;\mbox{if}\;\; \lambda=p-1, \;\;\;\; \mbox{and} \;\;\;\;  2p-1 \geq \;  D(\varphi , p) \; \geq p- \lambda+1,\;\;\;\mbox{otherwise}. 
\]
In particular, for $p \ge 5$ we have:
\begin{enumerate}
\item[(a)] if $\lambda = p-2$ then $3 \le  D(\varphi , p) \le \frac{k-1}{k}(p-1)+1;$
\item[(b)] if $\lambda = p-3$ then $4 \le  D(\varphi , p) \le \frac{2k-2}{k}(p-1)+1,$
\end{enumerate}
where $k$  is the smallest positive integer greater than $2$ such that $k\mid p-1.$
\end{theorem}
We are going to divide the proof into two lemmas, taking into account the possible values of $\lambda$ (see \eqref{mu=0}).

\begin{lemma}\label{lemma2.5bis}
Let $p$ be an odd prime, then 
\[
D(\varphi , p)=1, \;\;\mbox{if}\;\; \lambda=p-1, \;\;\; \mbox{and} \;\;\;  2p-1\ge   D(\varphi , p)\geq p- \lambda+1,\;\;\;\mbox{otherwise}. 
\]
\end{lemma}
\begin{proof}
The upper bound follows from Theorem \ref{Principal} (iv).
Next, observe that if  $S=[u]^{k}$ then it follows from  \eqref{mu=0} that
\begin{equation}\label{PL}
		\varphi(S) = (uk)^{2} + \lambda ku^{2} = u^{2}k(k + \lambda ).
\end{equation}
	
	If  $\lambda = p-1$ then $\varphi([u]) = u^{2}-u^{2}=  0$ (see \eqref{PL}), hence any sequence has a $\varphi$ - zero subsequence, thus $ D(\varphi , p)=1$.  For  $\lambda \neq p-1$,  (see \eqref{PL}), we can have $\varphi([u]^{k})\neq 0$,  as long as $k \leq (p- \lambda) -1$. 
Now, consider the sequence  $S=[u]^{p-\lambda-1}[tu]$ for some $u,t \in \mathbb{F}_q^*$ with $t\neq 1$.
From the considerations above, we have $\varphi(T)\neq 0$ for the subsequence $T=[u]^k,$ with $1 \le k \le p-\lambda -1.$
Let us consider now the subsequence $T=[u]^k[tu].$ In this case
\[
\varphi(T)= (ku + tu)^2 + \lambda (ku^2 +t^2 u^2)=u^2 (k^2 + (2t+\lambda)k + (\lambda + 1)t^2).
\]
Observe that if $\Delta(t)= (2t+\lambda)^2 - 4(\lambda+1)t^2$ is not a square in $\mathbb{F}_p$, then $\varphi(T)\neq 0$ for any $k \in \mathbb{F}_p,$  and observe also that $\lambda \neq 0$ and $\lambda+1\neq 0,$ since $1\le  \lambda \le p-2.$
If $-(\lambda+1)$ is not a square,
then $-(\lambda+1)\neq 1$ and so $-2^{-1}\lambda \neq 1.$ Thus,
it suffices to choose $t=-2^{-1}\lambda\neq 1,$ since
$\Delta(-2^{-1}\lambda)=-(\lambda+1)\lambda^2.$ 
So  let us consider that there exists $r\in \mathbb{F}_p^*$ such that
$r^2=-(\lambda+1).$
Since there exist  non-quadratic residues, there must exist $s\in \mathbb{F}_p^*$ such that $s^2+1$ is not a square. Since  $(-s)^2+1$  is also not a square, we can  choose this $s$  in such way  that  $rs \neq 1$.  Rewrite $\Delta(t)$ as
\[
\Delta(t)=(2t+\lambda)^2 + (2rt)^2=(2rt)^2(((2t+\lambda)(2rt)^{-1})^2+1).
\]
Observe that $\lambda (2rs-2)^{-1}\neq 1,$ since if $\lambda (2rs-2)^{-1}=1$ then $\lambda=2rs-2$. But from
$r^2=-(\lambda +1)$, we get $(r+s)^2=s^2+1$. This is not possible, since $s^2+1$ is not a square.
So, we may choose
$t=\lambda (2rs-2)^{-1}\neq 1$ and we have 
\[
(2t+\lambda)(2rt)^{-1}=s \;\;\;\mbox{ and } \;\;\; \Delta(t)=(2rt)^2(s^2+1).
\]
Since $\Delta(t)$  is not a square, this implies that $S=[u]^{p-\lambda-1}[tu]$ is a $\varphi$-zero free sequence and
$D(\varphi , p)\geq p- \lambda+1.$

\end{proof}

\begin{lemma}\label{L2.6}
Let $p$ be an odd prime such that $p\ge 5,$ then 
\[
D(\varphi , p)
\le
\left\{
\begin{matrix}
\dfrac{k-1}{k}(p-1)+1 & \textrm{if } \lambda = p-2;\\ \\
\dfrac{2k-2}{k}(p-1)+1 & \textrm{if } \lambda = p-3,
\end{matrix}
\right. 
\]
where $k$
is the smallest positive integer greater than $2$ such that $k\mid p-1.$
\end{lemma}
\begin{proof} 
Let $\alpha\in \mathbb{F}_p^*$ be a generator of $\mathbb{F}_p^{*}$ and let
$\beta=\alpha^{\frac{p-1}{k}}.$   Since $k>2$, we have $\beta^2\ne 1$ and for all $a \in \mathbb{F}_p^*$
\vspace{-0.2cm}
\begin{equation}\label{beta}
\sum_{j=0}^{k-1} a \beta^j = a \frac{\beta^k-1}{\beta-1} = 0 \mbox{ and } \sum_{j=0}^{k-1} (a \beta^j)^2 = a^2 \frac{\beta^{2k}-1}{\beta^2-1} = 0.
\end{equation}

Hence, a $\varphi$-zero free sequence  $S=[v_1]^{d_{1}}[ v_{2}]^{d_{2}} \cdots [v_{m}]^{d_{m}}$  can not have  a subsequence of the form  $V=[a][a\beta]\cdots [a\beta^{k-1}]$,  since  $\varphi(V)=0$ (see \eqref{beta}).
From the considerations following \eqref{PL}, if $\lambda = p-2$ , we must have $d_{1}=\cdots=d_{m}=1$, so the $\varphi$-zero free sequence  $S$  contains at most $k-1$ elements of each lateral class of the group generated by $\beta.$ Since there are $\frac{p-1}{k}$ lateral  classes we must have \vspace{-0.1cm}\[D(\varphi , p)  \le (k-1)\frac{p-1}{k} + 1.\]\vspace{-0.15cm} 
By a similar reasoning, if  $\lambda = p-3$ then  we must have, for any $j\in\{1,2,\ldots, m\}$, $d_{j}\leq 2$,  hence \[D(\varphi , p)  \le 2(k-1)\frac{p-1}{k} +  1,\] \vspace{-0.3cm} completing the proof.
\end{proof}
A computer search using Sagemath \cite{SAGE}, gives us Table \ref{table_sec2}
containing the exact value for $D(\varphi ,p)$, for $F=x_1^2+\lambda x_2$, $3 \le p\leq 11$
and $1 \le \lambda \le p-2,$ and also presenting examples of  extremal $\varphi$-zero free sequences and the cardinality of the set $M(\varphi, p)$. Observe that, when $\lambda = p-2$,  the upper bound given in Lemma \ref{L2.6} is attained when $p=5$
and $p=7,$ but as $p$ grows (bigger than 17, for example), a computer search reveals that  the exact value of $D(\varphi ,p)$ tends to be smaller than $p/2$.
\begin{table}[H]
\centering
{\small
\begin{tabular}{cccc}
\hline\noalign{\smallskip}
\multirow{2}{*}{$F$} &
\multirow{2}{*}{$D(\varphi,p)$} & 
\multirow{2}{*}{Some elements of $M(\varphi,p)$} &
\multirow{2}{*}{$|M(\varphi,p)|$}\\
& & & \\
\noalign{\smallskip}\hline\noalign{\smallskip}
\vspace{0.1cm} $x_1^2+x_2$  & $D(\varphi,3)=3$ & $[1][2]$ & $1$\\
\vspace{0.1cm} $x_1^2+x_2$  & $D(\varphi,5)=7$ & $[1]^3[4]^3, \;\;[2]^3[3]^{3}$ & $2$ \\
\vspace{0.1cm} $x_1^2+2x_2$ & $D(\varphi,5)=6$  &$[1]^2[2]^2[4],\;\;[1]^2[2][3]^2$ & $4$\\  
\vspace{0.1cm} $x_1^2+3x_2$ & $D(\varphi,5)=4$ & $[1][2][3],\;\;[1][2][4]$ & $4$\\
\vspace{0.1cm} $x_1^2+x_2$  & $D(\varphi,7)=7$ &$[1]^5[3],\;\;[1]^3[6]^3$ & $15$\\
\vspace{0.1cm} $x_1^2+2x_2$ & $D(\varphi,7)=8$ &$[1]^4[5]^3,\;\;[2]^4[3]^3$ & $6$\\
\vspace{0.1cm} $x_1^2+3x_2$ & $D(\varphi,7)=8$ &$[1]^3[3]^3[5],\;\;[2]^2[4]^3[6]^2$ & $12$\\
\vspace{0.1cm} $x_1^2+4x_2$ & $D(\varphi,7)=5$ &$[1]^2[2]^2,\;\;[3]^2[4]^2$ & $9$\\
\vspace{0.1cm} $x_1^2+5x_2$ & $D(\varphi,7)=5$ &$[1][2][3][5],\;\;[2][3][4][6]$ & $9$\\
\vspace{0.1cm} $x_1^2+x_2$  & $D(\varphi,11)=13$&$[1]^9[3]^3,\;\;[2]^9[6]^2[7]$ & $40$\\
\vspace{0.1cm} $x_1^2+2x_2$ & $D(\varphi,11)=12$ &$[1]^8[4]^3,\;\;[2]^8[5]^2[8]$ & $50$\\
\vspace{0.1cm} $x_1^2+3x_2$ & $D(\varphi,11)=12$ &$[1]^6[7]^2[8]^3,\;\;[2]^2[5]^6[7]^3$ & $10$\\
\vspace{0.1cm} $x_1^2+4x_2$ & $D(\varphi,11)=11$ &$[1]^6[6]^4,\;\;[3]^5[8]^5$ & $15$\\
\vspace{0.1cm} $x_1^2+5x_2$ & $D(\varphi,11)=11$ &$[1]^5[2]^5,\;\;[3]^5[7]^5$ & $10$\\
\vspace{0.1cm} $x_1^2+6x_2$ & $D(\varphi,11)=11$ &$[1]^4[4]^3[10]^3,\;\;[4]^4[5]^3[7]^3$ & $10$\\
\vspace{0.1cm} $x_1^2+7x_2$ & $D(\varphi,11)=9$ &$[1]^3[5]^3[7]^2,\;\;[4]^3[6]^2[9]^3$ & $10$\\
\vspace{0.1cm} $x_1^2+8x_2$ & $D(\varphi,11)=8$ &$[1]^2[2]^2[3][5][6],\;\;[3][6]^2[7]^2[8]^2$ & $60$\\
\vspace{0.1cm} $x_1^2+9x_2$ & $D(\varphi,11)=6$ &$[1][2][3][4][7],\;\;[4][5][6][7][9]$ & $60$\\
\noalign{\smallskip}\hline
\end{tabular}
}
\caption{\footnotesize{Exact Values of $D(\varphi,p)$ for $F=x_1^2+\lambda x_2$}}
\label{table_sec2}
\end{table}

\section{ The Polynomial  $\varphi(F,m) = as_{m,1}^{2} +bs_{m,2} + cs_{m,1}$, with $abc\neq 0$}

With no loss in generality we will rewrite the symmetric polynomial $\varphi(F,m)$ as
\begin{equation}\label{phiF}
\varphi(F,m) = s_{m,1}^{2} +\lambda s_{m,2} + \mu s_{m,1}, \;\;\; \mbox{with}\;\; \lambda \mu \not \equiv 0 \pmod{p},
\end{equation}
and for any  $S=[ u_{1}]^{t_{1}}\cdots [ u_{r}]^{t_{r}}$ we have 
\begin{equation}\label{phiS}
\varphi(S) \equiv (\;\sum_{i=1}^{r} t_{i}u_{i}  \;)^{2} + \sum_{i=1}^{r} t_{i}u_{i}(\lambda u_{i}+\mu) \pmod{p}.
\end{equation}

Let us define
\begin{equation} \label{omega}
\om = -\mu\cdot\lambda^{-1}.
\end{equation}
\begin{lemma}\label{omegaMax}
The sequence  $[\om]^{p-1}$ is an extremal $\varphi$-zero free sequence among all the sequences of the form $[u]^{k}$.
\end{lemma}
\begin{proof}
We have
\begin{equation}\label{phi-omega}
\varphi([\om]^{t}) \equiv (t\mu\lambda^{-1})^{2} +\lambda t (\mu\lambda^{-1})^{2} - t\mu(\mu\lambda^{-1}) \equiv  (t\mu\lambda^{-1})^{2} \not \equiv 0  \pmod{p},
\end{equation}
for any $t\in\{1,2,3,\ldots, p-1\}$. Hence $S=[\om]^{p-1}$ is a $\varphi$-zero free sequence, and also observe that for any other value $u\neq\om$ there is a $t_{o}\in\{1,2,3,\ldots, p-1\}$ such that 
\[
 (t_{o} + \lambda) u + \mu \equiv 0 \pmod{p},
\]
in particular 
\[
\varphi([u]^{t_{o}}) = t_{o}u(t_{o}u + \lambda  u + \mu)\equiv 0 \pmod{p},
\]
completing the proof.
\end{proof}
\begin{lemma} \label{p=3}
 Let $p=3$ and consider \eqref{phiF}, then
\begin{enumerate}
\item  $D(\varphi, 3)=3\;\;$  if $\;\;\lambda = 1;$
\item $D(\varphi, 3)=4\;\;$  if  $\lambda = 2.$ 

\end{enumerate}
\end{lemma}
\begin{proof} Let $u\in\{1,2\}$.  Since 
\[
\varphi([u]^{k}) = ku(ku + \lambda u + \mu) \equiv 0 \;\;\mbox{if}\;\;\; k\equiv 0 \pmod{3},
\]
any $\varphi$-zero free sequence in $\mathbb{F}_{3}$ has the form $[1]^{r}[2]^{s}$ with $0\leq r,s \leq 2$ and $(r,s)\neq (0,0)$. 
For $\lambda = 1$ we have
\[
\varphi([\mu]) = 0, \;\;\;\mbox{and}\;\;\;\varphi([2\mu]^{2}) = 1,
\]
so an extremal $\varphi$-zero free sequence is $[2\mu]^{2}.$

For $\lambda = 2$ we have
\[
\begin{array}{ll}
 \varphi([\mu])=\varphi([\mu]^{2}) = \mu^{2} = 1,  & \varphi([2\mu]) = 2\mu^{2} = 2, \\
\varphi([\mu]^{2}[2\mu]) = 2\mu^{2} = 2, &  \varphi([2\mu]^{2}) = 0.
\end{array}
\]
Thus, an extremal $\varphi$-zero free sequence is $[\mu]^{2}[2\mu]$.
\end{proof} 
\begin{lemma}\label{L1}
Let $p$ be a prime number such that $p\ge 5.$ Then, there exists an $u^{*}\in\Fp^{*}$,  such that $-u^{*}(\lambda u^{*}+\mu)$ is not a quadratic residue modulo $p.$ 
\end{lemma}
\begin{proof} 
First observe that  we can always find $z_{0}\in \Fp^{*}$ and $z_{1}\in\{1,2,3,4\}$ such that
\[
\left ( \frac{z_{o}}{p} \right ) = - \; \left ( \frac{z_{o}+ 1}{p} \right )\;\;\;\mbox{and}\;\;\; \left ( \frac{z_{1}}{p} \right ) = \left ( \frac{z_{1}+ 1}{p} \right ).
\]
Considering $j =\lambda \mu^{-1}u$, we have
\[
\left ( \frac{-u(\lambda u+\mu)}{p} \right ) = \left ( \frac{-\lambda}{p} \right )\cdot  \left ( \frac{\mu^{2}}{p} \right ) \cdot \left ( \frac{j}{p} \right ) \cdot \left ( \frac{j+1}{p} \right ),
\]
since 
\[
-\,u(\lambda u+\mu) \; = - \, \lambda^{-1}\mu^{2}.(\lambda \mu^{-1}u).(\lambda \mu^{-1}u +1) \; = \; - \, \lambda^{-1}\mu^{2}.(j.(j +1)).
\]
Hence,  according to $-\lambda$ being or not a quadratic residue modulo $p$,  we can choose $j\in\{z_{0}, \,z_{1}\}$, in such a way that  $-u(\lambda u+\mu)$ is not a quadratic residue modulo $p.$ 
\end{proof}

\begin{theorem}\label{lb}
Let $p\ge 5$ be a prime number and $s$ a positive integer. If for all $i \in \{1,\dots,s\},$ we have $\left(\frac{i}{p}\right)=1,$ then $2p-1 \ge D(\varphi,p)\ge p+s.$
\end{theorem}
\begin{proof}
The upper bound follows from Theorem \ref{Principal} (iv).
Consider the sequence $S=[\om]^{p-1}[u^{*}]^{s}$, with $u^{*}$ given in Lemma \ref{L1} and $\om$ defined in \eqref{omega}. Since $\om(\lambda \om + \mu) =0$, we have  $u^{*} \neq \om$. Let $T=[\om]^{t}[u^{*}]^{i}$ be any subsequence of $S$. Hence (see \eqref{phiS})
\[
\varphi(T) =  (t\om + i u^{*})^{2} + i(u^{*}(\lambda u^{*}+\mu)).
\]
According to the hypothesis and Lemma \ref{L1} we have  
\[
\left(\dfrac{- i(u^{*}(\lambda u^{*}+\mu))}{p}\right)=\left(\dfrac{ i}{p}\right)\cdot \left(\dfrac{- u^{*}(\lambda u^{*}+\mu))}{p}\right)= -1.
\]
Hence  $S$ is a $\varphi$-zero free sequence, thus $D(\varphi,p)\ge (p-1) + s + 1$ as desired.
\end{proof}
At this point, since $\om$ has the property described in Lemma \ref{omegaMax}, one would expect to always find a $\varphi$-zero free sequence of the form $[\om]^{p-1}[u_1]^{k_1}\cdots[u_t]^{k_t}$ in the set $M(\varphi, p)$. But thus far we do not have an answer for this supposition, although if this is the case one would obtain a better upper bound for the constant $D(\varphi,p)$. 
\begin{theorem}\label{cota}
If among the extremal $\varphi$-zero  free sequences in the set $M(\varphi, p)$ we can find a sequence of the type 
$S=[\om]^{p-1}[u_1]^{k_1}\cdots[u_t]^{k_t},$  with $\om$  as defined in \eqref{omega},  then \[D(\varphi,p) \leq (p-1) + \dfrac{p-1}{2}.\]

\end{theorem}
\begin{proof}
After reordering the indexes, we can assume that $k_1\le\cdots\le k_t.$  Let $T=[\om]^{j}[u_1]^{i_1}\cdots[u_t]^{i_t}$ be any subsequence of $S$. Since we are assuming that $S$ is a $\varphi$-zero free sequence, we must have  $\varphi(T)\not\equiv 0 \pmod p.$ 

Let $\mathcal{N}=\left\{(\ell_1,\dots,\ell_t); \;\; 0\le \ell_r\le k_r,\;  r=1,\dots,t\right\}$, then we can write  
\begin{equation}
\varphi(T)=A(j,I)^2-B(I).
\end{equation}
with $I=(i_1,\dots, i_t)\in \mathcal{N}$ and 
\begin{equation}\label{ABs}
B(I)= -\displaystyle\sum_{r=1}^{t}i_rv_r(v_r+1)\;\;\;\mbox{and}\;\;\;
A(j,I)=j\om+\displaystyle\sum_{r=1}^{t}i_rv_r.
\end{equation}

Let $\mathcal{N}^{\ast}=\mathcal{N}\backslash \{(0,\dots,0)\}$. Note that for any  $I\in \mathcal{N}^{\ast},$ we must have
\begin{equation}\label{eq.11}
\left(\dfrac{B(I)}{p}\right)=-1,
\end{equation}
otherwise $B(I)\equiv 0 \pmod p$ or $\left(\dfrac{B(I)}{p}\right)=1,$ which means that we could find $j \in \{0,1,2,...,p-1\}$ such that $A(j,I)^2\equiv B(I) \pmod p.$ In that case we would have found two $\varphi$-zero subsequences 
\[T_{0} = [\om]^{j_{0}}[u_1]^{i_1}\cdots[u_t]^{i_t} \mbox{ and } T_{1}=[\om]^{j_{1}}[u_1]^{i_1}\cdots[u_t]^{i_t}\]
of $S$ that is a $\varphi$-zero free sequence, a contradiction.

Now define the following disjoint subsets of $\mathcal{N}^{\ast}:$ 
\begin{eqnarray*}
\mathcal{I}_{1} &=& \{(1,0,\dots,0),(1,1,0,\dots,0),\dots, (1, 1,\dots,1)\}\\
\mathcal{I}_{2} &=& \{(2,1,\dots,1),(2,2,1,\dots,1),\dots, (2, 2,\dots,2)\}\\
\mathcal{I}_{3} &=& \{(3,2,\dots,2),(3,3,2,\dots,2),\dots, (3, 3,\dots,3)\}\\
      &\vdots &\\
\mathcal{I}_{k_t} &=& \{(k_t,k_t-1,\dots,k_t-1),\dots, (k_t,k_t,\dots,k_t)\}\\
\mathcal{I}_{k_t+1} &=& \{(k_t+1,k_t,\dots,k_t),\dots, (k_t+1,k_t+1,\dots,k_t)\}\\
      &\vdots &\\
\mathcal{I}_{k_{t-1}} &=& \{(k_{t-1},k_{t-1}-1,\dots,k_{t-1}-1,k_t),\dots, (k_{t-1},k_{t-1},\dots, k_{t-1}, k_t)\}\\
\mathcal{I}_{k_{t-1}+1} &=& \{(k_{t-1}+1,k_{t-1},\dots,k_{t-1},k_t),\dots, (k_{t-1}+1,k_{t-1}+1,\dots, k_{t-1}+1, k_{t-1}, k_t)\}\\
      &\vdots &\\
\mathcal{I}_{k_{2}} &=& \{(k_{2},k_{2}-1, k_3, \dots, k_t), (k_{2},k_{2}, k_3, \dots, k_t)\}\\
\mathcal{I}_{k_{2}+1} &=& \{(k_{2}+1,k_{2}, k_3, \dots, k_t)\}\\
			&\vdots &\\
\mathcal{I}_{k_{1}} &=& \{(k_{1},k_{2}, k_3, \dots, k_t)\}\\			
\end{eqnarray*}

Let $\mathcal{J}=\mathcal{I}_{1}\cupdot \mathcal{I}_{2}\cupdot\cdots\cupdot \mathcal{I}_{k_1}$ and observe that  $\mathcal{J}\subset \mathcal{N}^{\ast}.$ Also note that the cardinality of this set $\mathcal{J}$ is equal to

\begin{eqnarray*}
|\mathcal{J}| & = & tk_t+(t-1)(k_{t-1}-k_t)+(t-2)(k_{t-2}-k_{t-1})+ \cdots +(t-j)(k_{t-j}-k_{t-j+1})\\
    &   & + \cdots +2(k_2-k_3)+k_1-k_2\\
    & = &k_{1}+k_{2}+ \cdots + k_{r}.
\end{eqnarray*}

Let $I_0=(i_1,\dots, i_t), \, I_1=(i_1^{*}, \dots, i_t^{*})\in \mathcal{J}$.  By the formation of $\mathcal{J}$,  we may always assume $i_s\le i_s^{*},$ for all $s\in\{1,\dots, t\}$,  and for at least one $r\in\{1,\dots, t\}$ we have $i_r < i_r^{*}.$  Thus,

\[
I^{*}=(i_1^{*}-i_1, \dots, i_r^{*}-i_r, \dots, i_t^{*}-i_t) \in \mathcal{N}^{\ast}
\]
and therefore, according to \eqref{eq.11}, we must have 
\begin{equation}
B(I^{*})\equiv B(I_1) - B(I_{0})\not\equiv 0 \pmod p,
\end{equation} 
since   $B(I)$ is a linear form in the indexes $i_{j}$ (see \eqref{ABs}). Therefore the values of $B(I)$ are all distinct for any $I\in\mathcal{J}$, and also $\left(\frac{B(I)}{p}\right)=-1,$ for all $I\in \mathcal{J},$ (see \eqref{eq.11}). The conclusion is that 
\[
|\mathcal{J}|=k_{1}+k_{2}+ \cdots + k_{r}\le \frac{p-1}{2}.
\]
Therefore, the length of the sequence $S$ is at most
\[
|S|=p-1+k_{1}+k_{2}+\cdots + k_{r}\le p-1 + \frac{p-1}{2}.
\]
\end{proof}
\section{ The Special Case of  $\varphi(F,m) = s_{m,1}^{2} +s_{m,2} + s_{m,1}$}

We want to present some results and comments about this special case, as a way of shedding more light on the bounds given in Theorems \ref{lb} and \ref{cota}. Let us start with the corollary below, which  is a straightforward consequence of the {\it Quadratic Reciprocity Law} and Theorem \ref{lb}.

\begin{corollary} Let $p$ be an odd prime. Then
\begin{enumerate}
\item [(i)] $D(\varphi,p)\geq p+2\;$ if $\;p\equiv \pm 1  \pmod{24};$
\item [(ii)] $D(\varphi,p)\geq p+4\;$ if $\;p\equiv \pm 1  \pmod{60};$
\item [(iii)] $D(\varphi,p)\geq p+6\;$ if $\;p\equiv\pm 1 \, \mbox{or} \, \pm 49  \pmod{120}.$
\end{enumerate}
\end{corollary}

Using these ideas it is easy to obtain the following table containing lower bounds for the Davenport $\varphi$-constant  $D(\varphi ,p)$:
\begin{table}[H]
		\centering
		{\small
			\begin{tabular}{lll}
				\hline\noalign{\smallskip}
				\noalign{\smallskip}\hline\noalign{\smallskip}
				\vspace{0.1cm}  $D(\varphi,5)\geq 6$  && $D(\varphi,7)\geq 9 $ \\
				\vspace{0.1cm}  $D(\varphi,11)\geq 12$ &&  $D(\varphi,13)\geq 14$  \\
				\vspace{0.1cm}  $D(\varphi,17)\geq 19$ &&  $D(\varphi,23)\geq 27$\\
				\vspace{0.1cm}  $D(\varphi,29)\geq 30$ &&   $D(\varphi,31)\geq 33$ \\
				\vspace{0.1cm}  $D(\varphi,71)\geq 77$ &&   $D(\varphi,311)\geq 321$\\
										
				\noalign{\smallskip}\hline
			\end{tabular}
		}
		\caption{\footnotesize{Bounds for $D(\varphi,p)$}}
		\label{table_ex_p(5)}
	\end{table} 

On the other hand, a computer search using Sagemath \cite{SAGE}, gives us the table below containing the exact value for $D(\varphi ,p)$, for $p\leq 31$, and also presenting  extremal $\varphi$-zero free sequences, and the cardinality of the set $M(\varphi, p)$.

\begin{table}[H]
	\centering
	{\small
\begin{tabular}{ccc}
\hline\noalign{\smallskip}
\multirow{2}{*}{$D(\varphi,p)$} & \multirow{2}{*}{Some elements of $M(\varphi,p)$} &
\multirow{2}{*}{$|M(\varphi,p)|$}\\
&  & \\
\noalign{\smallskip}\hline\noalign{\smallskip}
\vspace{0.1cm}  $D(\varphi,3)=3$ &  $[-1]^{2}$ & $1$\\
\vspace{0.1cm}  $D(\varphi,5)=6$ &  $[1][-1]^{4},[3][-1]^{4}$ & $2$ \\
\vspace{0.1cm}  $D(\varphi,7)=9$ &  $[1]^2[-1]^{6},[1][5][-1]^{6},[5]^2[-1]^{6}$ & $3$\\  
\vspace{0.1cm}  $D(\varphi,11)= 12$ &  $[1]^8[2]^3,[3][-1]^{10}$ & $8$\\
\vspace{0.1cm}  $D(\varphi,13)=15$ &  $[1]^{10}[7]^4,[1][2][-1]^{12}$ & $6$\\
\vspace{0.1cm}  $D(\varphi,17)=19$ &  $[2]^2[-1]^{16},[2][7][-1]^{16}$ & $16$\\
\vspace{0.1cm}  $D(\varphi,19)=21$ &  $[1]^{16}[3]^4,[2][4][-1]^{18}$ & $22$\\
\vspace{0.1cm}  $D(\varphi,23)=27$ &  $[1]^4[-1]^{22},[1]^3[-2][-1]^{22}$ & $25$\\
\vspace{0.1cm}  $D(\varphi,29)=31$&  $[1]^{26}[12][13]^3,[6]^{12}[10]^{18},[1][3][-1]^{28}$ & $54$\\
\vspace{0.1cm}  $D(\varphi,31)=34$&  $[1][11]^2[-1]^{30},[7][8][22][-1]^{30}$ & $54$\\
\noalign{\smallskip}\hline
		\end{tabular}
	}
	\caption{\footnotesize{Exact Values of $D(\varphi,p)$}}
	\label{table_ex_p(6)}
\end{table} 

Observe that in the examples of extremal $\varphi$-zero free sequences given in Table 2, we can always find a sequence with the term $[-1]^{p-1}$. This strengthens our initial supposition  that it is always possible to find  extremal $\varphi$-zero free sequences containing the term $[\om]^{p-1}$ (see \eqref{omega}), for any prime $p$. But even for this special case we  were not able to prove it. On the other hand, the tables above show us, that the exact value of $D(\varphi,p)$ seems to be closer to the lower bound given in Theorem \ref{lb}. In conclusion, there are many points here to be clarified, and we believe that these are questions worthy to be pursued.


\end{document}